\documentclass[11pt]{amsart}

\usepackage{amsmath,amsthm,amssymb,color}
\usepackage{pdfsync}
\usepackage[latin1]{inputenc}
\usepackage{graphicx}
\usepackage{pstricks,epic,eepic}

\newcommand{\R}{{\mathbb R}}       
\newcommand{\DD}{{\mathcal D}}
\newcommand{\HH}{{\mathcal H}}

\newcommand{\BZ}{{\mathcal B}}

\newcommand{\TT}{{\mathcal T}}

\newcommand{\diam}{\operatorname{diam}}
\newcommand{\dist}{{\rm dist}}

\newcommand{\rf}[1]{{\eqref{#1}}}
\newcommand{\supp}{\operatorname{supp}}

\newcommand{\ve}{{\varepsilon}}
\newcommand{\vv}{{\vspace{2mm}}}
\newcommand{\vvv}{{\vspace{3mm}}}
\newcommand{\wt}[1]{{\widetilde{#1}}}

\newcommand{\noi}{\noindent}

\newcommand{\nex}{{\mathsf{Next}}}

\newcommand{\ttt}{{\mathsf{Top}}}

\newtheorem{theorem}{Theorem}[section]
\newtheorem*{theorem*}{Theorem}
\newtheorem*{lemma*}{Lemma}
\newtheorem*{theorema*}{Theorem A}
\newtheorem*{theoremb*}{Theorem B}
\newtheorem*{theoremc*}{Theorem C}
\newtheorem*{mainlemma*}{Main Lemma}

\newtheorem{lemma}[theorem]{Lemma}

\theoremstyle{definition}

\def\XXint#1#2#3{{\setbox0=\hbox{$#1{#2#3}{\int}$ }
\vcenter{\hbox{$#2#3$ }}\kern-.58\wd0}}

\theoremstyle{remark}
\newtheorem{remark}[theorem]{\bf Remark}

\numberwithin{equation}{section}

\begin{document}

\title{Rectifiability of measures and the $\beta_p$ coefficients}


\author{Xavier Tolsa}
\address{Xavier Tolsa
\\
ICREA, Passeig Llu\'{\i}s Companys 23, 08010 Barcelona, Catalonia, and\\
Departament de Matem\`atiques and BGSMath
\\
Universitat Aut\`onoma de Barcelona
\\
08193 Bellaterra (Barcelona), Catalonia
}

\email{xtolsa@mat.uab.cat}

\thanks{The author was supported by the ERC grant 320501 of the European Research Council (FP7/2007-2013) and partially supported by MTM-2016-77635-P,  MDM-2014-044 (MICINN, Spain), 2014-SGR-75 (Catalonia), and by Marie Curie ITN MAnET (FP7-607647).
}

\begin{abstract}
In some former works of Azzam and Tolsa it was shown that $n$-rectifiability can be characterized in terms
of a square function involving the David-Semmes $\beta_2$ coefficients. In the present paper we 
construct some counterexamples which show that a similar characterization does not hold
for the $\beta_p$ coefficients with $p\neq2$. This is in strong contrast with what happens in the case of uniform
$n$-rectifiability. In the second part of this paper we provide an alternative argument for a recent result of Edelen, Naber and Valtorta about the $n$-rectifiability of measures with bounded lower $n$-dimensional density. Our alternative proof
follows from a slight variant of the corona decomposition in one of the aforementioned works of Azzam and Tolsa and a suitable approximation argument.
\end{abstract}

\maketitle

\section{Introduction}

Let $\mu$ be a Radon measure in $\R^d$.
One says that $\mu$ is $n$-rectifiable if
there are Lipschitz maps
$f_i:\R^n\to\R^d$, $i=1,2,\ldots$, such that 
\begin{equation}\label{eqdef00}
\mu\biggl(\R^d\setminus\bigcup_i f_i(\R^n)\biggr) = 0,
\end{equation}
and $\mu$ is absolutely continuous with respect to the $n$-dimensional Hausdorff measure $\HH^n$. 
A set $E\subset \R^d$ is called $n$-rectifiable if the measure $\HH^n|_E$ is $n$-rectifiable.
On the other hand, $E$ is called purely $n$-unrectifiable if any $n$-rectifiable subset $F\subset E$ 
satisfies $\HH^n(F)=0$.

The study of $n$-rectifiability of sets and measures is one of the main objectives of geometric measure theory.
The introduction of multiscale quantitative techniques by Jones \cite{Jones} in the 1990's was very fruitful and influential in the area of geometric analysis because its applications to other related questions, for example in connection with singular integrals and analytic capacity (see \cite{DS1}, \cite{DS2}, \cite{David-vitus}, \cite{Leger}, \cite{NToV}, or \cite{Tolsa-bilip}, for instance).

In the monograph \cite{DS1}, David and Semmes introduced the notion of uniform $n$-rectifiability,
which should be considered as a quantitative version of $n$-rectifiability.
Let $\mu$ be an $n$-AD-regular (i.e., $n$-Ahlfors-David regular) Radon measure, that is, for some constant $c>0$, 
\begin{equation}\label{eq-ad}
c^{-1}r^n\leq \mu(B(x,r))\leq c\,r^n\quad \mbox{ for all $x\in
\supp\mu$ and $0<r\leq \diam(\supp\mu)$.}
\end{equation}
The measure $\mu$ is called uniformly  $n$-{{rectifiable}} if, besides being 
$n$-AD-regular,
there exist constants $\theta, M >0$ such that for all $x \in \supp\mu$ and all $0<r\leq \diam(\supp\mu)$ 
there is a Lipschitz mapping $g$ from the ball $B_n(0,r)\subset\R^{n}$ to $\R^d$ with $\text{Lip}(g) \leq M$ such that
$$
\mu(B(x,r)\cap g(B_{n}(0,r)))\geq \theta \,r^{n}.$$
A set $E\subset\R^d$ is called uniformly $n$-{{rectifiable}} if $\HH^n|_E$ is uniformly $n$-{{rectifiable}}.

In  \cite{DS1} and \cite{DS2}, David and Semmes gave several equivalent characterizations of uniform $n$-rectifiability. 
One of the most relevant involves the $\beta_{p}$ coefficients.
For $1\leq p <\infty$, $x\in\R^d$, $r>0$, one defines
$$\beta_{\mu,p}^n(x,r) = \inf_{\text{$L\subset\R^d$ is an $n$-plane}} \left(\int_{B(x,r)} \left(\frac{\dist(y,L)}r\right)^p\,
\frac{d\mu(y)}{r^n}\right)^{1/p},$$
and also
$$\wt\beta_{\mu,p}^n(x,r) = \inf_{\text{$L\subset\R^d$ is an $n$-plane}} \left(\int_{B(x,r)} \left(\frac{\dist(y,L)}r\right)^p\,
\frac{d\mu(y)}{\mu(B(x,r)}\right)^{1/p},$$
It was shown in \cite{DS1} that, for $1\leq p <2n/(n-2)$ in the case $n> 2$ and $1<p<\infty$ in the case $n=1$ or $2$, an $n$-AD-regular measure $\mu$ in $\R^d$ is uniformly $n$-rectifiable if
and only if 
\begin{equation}\label{eqds1}
\int_{B(z,R)} \int_0^{R}\beta_{\mu,p}^n(x,r)^2\,\frac{dr}r\,d\mu(x)\leq c\,R^n \quad\mbox{ for all $z\in\supp\mu$, $R>0$.}
\end{equation}
Of course, the same statement is valid replacing the coefficients $\beta_{\mu,p}^n(x,r)$ by $\wt\beta_{\mu,p}^n(x,r)$,
because they are comparable for all $x\in\supp\mu$, $0<r\leq\diam(\supp\mu)$ when $\mu$ is
$n$-AD-regular.

More recently,  Jonas Azzam and the author obtained a related characterization of $n$-rectifiability for general Radon measures with positive and bounded upper $n$-dimensional density.
The upper and lower $n$-dimensional densities of $\mu$ at a point $x\in\R^d$ are defined, respectively, by
$$\Theta^{n,*}(x,\mu)= \limsup_{r\to 0}\frac{\mu(B(x,r))}{(2r)^n},\qquad
\Theta^{n}_*(x,\mu)= \liminf_{r\to 0}\frac{\mu(B(x,r))}{(2r)^n}.$$
The aforementioned characterization of $n$-rectifiability is the following:

\begin{theorema*}
Let $\mu$ be a Radon measure in $\R^d$ 
such that $0<\Theta^{n,*}(x,\mu)<\infty$ for $\mu$-a.e.\ $x\in\R^d$. Then $\mu$ is $n$-rectifiable if
and only if
\begin{equation}\label{eqjones*}
\int_0^1 \beta_{\mu,2}^n(x,r)^2\,\frac{dr}r<\infty \quad\mbox{ for $\mu$-a.e.\ $x\in\R^d$.}
\end{equation}
\end{theorema*}
\vv

The ``if'' direction of the theorem was proven in \cite{Azzam-Tolsa-GAFA}, and the
``only if'' one in \cite{Tolsa-CVPDE}. As an immediate corollary of the above result, it follows
that a set $E\subset\R^d$ with $\HH^n(E)<\infty$ is $n$-rectifiable if and only if
$$\int_0^1 \beta_{\HH^n|_E,2}^n(x,r)^2\,\frac{dr}r<\infty \quad\mbox{ for $\HH^n$-a.e.\ $x\in E$.}$$

For other criteria for rectifiability in terms of related square functions which apply to measures which  are absolutely continuous with respect to Hausdorff measure,  see \cite{Lerman} or  \cite{Tolsa-memo}, for example. 
For other recent works which study the rectifiability of more general measures, we refer the reader to 
\cite{Badger-Schul1}, \cite{Badger-Schul}, \cite{MO}, \cite{ADT}, or \cite{ATT}.

We also remark that quite recently Edelen,  Naber and Valtorta \cite{ENV} showed that
Theorem A also holds for Radon measures $\mu$ satisfying the conditions
\begin{equation}\label{eqenv*}
\Theta^{n,*}(x,\mu)>0 \quad \text{ and } \quad \Theta^{n}_*(x,\mu)<\infty \quad \text{ for $\mu$-a.e.
$x\in\R^d$,}
\end{equation}
instead of the more restrictive one 
\begin{equation}\label{eqenv**}
0<\Theta^{n,*}(x,\mu)<\infty \quad \text{ for $\mu$-a.e.
$x\in\R^d$.}
\end{equation}
That is, they proved the following:

\begin{theoremb*}[\cite{ENV}]
Let $\mu$ be a finite Borel measure in $\R^d$ 
satisfying \rf{eqenv*}.
If
\begin{equation}\label{eqjones**}
\int_0^1 \beta_{\mu,2}^n(x,r)^2\,\frac{dr}r<\infty \quad\mbox{ for $\mu$-a.e.\ $x\in\R^d$,}
\end{equation}
then $\mu$ is $n$-rectifiable.
\end{theoremb*}

Notice that the condition \rf{eqenv**} ensures that the measure $\mu$ is absolutely continuous
with respect to $\HH^n$, while the 
 the assumption \rf{eqenv*} does not. However, Theorem B implies that if both \rf{eqenv**} and \rf{eqjones**} hold, then $\mu$ is absolutely continuous
with respect to $\HH^n$. This is the main novelty in Theorem B.
 
\vv
In view of the characterization of uniform $n$-rectifiability in terms of the 
$\beta_{\mu,p}^n$ coefficients with $1\leq p <2n/(n-2)$ by David and Semmes described in \rf{eqds1}, it is natural
to think that perhaps Theorem A is also valid with $\beta_{\mu,2}^n$ replaced by $\beta_{\mu,p}^n$ for
some $p\neq2$. Under some additional assumptions on the measure $\mu$, the following result is already known:

\begin{theoremc*}
Let $1\leq p <2n/(n-2)$ in the case $n>2$, and $1\leq p < \infty$ in the case $n=1$ or $2$.
Let $\mu$ be a Radon measure in $\R^d$. The following hold:
\begin{itemize}
\item[(a)] Suppose that $0<\Theta^{n}_*(x,\mu)\leq \Theta^{n,*}(x,\mu)<\infty$ for $\mu$-a.e.\ $x\in\R^d$. If 
\begin{equation}\label{eqbeta*11}
\int_0^1 \beta_{\mu,p}^n(x,r)^2\,\frac{dr}r<\infty \quad\mbox{ for $\mu$-a.e.\ $x\in\R^d$,}
\end{equation}
then $\mu$ is $n$-rectifiable.

\item[(b)]  Suppose that $\mu=\HH^n|_E$ for some $E\subset \R^n$ and that $\mu$ is $n$-AD-regular. Then \rf{eqbeta*11} holds.
\end{itemize}

\end{theoremc*}

The statement (a) of the theorem was first proved by Pajot \cite{Pajot} in the particular case where $\mu=\HH^n|_E$, with $E\subset\R^d$ such that $\HH^n(E)<\infty$. Later on Badger and Schul \cite{Badger-Schul-pams} showed that this extends easily to any measure $\mu$ 
such that $0<\Theta^{n}_*(x,\mu)\leq \Theta^{n,*}(x,\mu)<\infty$ $\mu$-a.e. The statement (b) is also due to Pajot \cite{Pajot}.

We remark that the assumptions that $\Theta^{n}_*(x,\mu)>0$ $\mu$-a.e.\ in (a) and the fact that $\mu$ is $n$-AD-regular in (b) play an essential role in the arguments for the previous theorem. In fact, they allow to reduce the arguments  to the case when the measure $\mu$ is $n$-AD-regular and to apply then the
result of David and Semmes stated in \rf{eqds1}.\vv

For arbitrary Radon measures $\mu$, from H\"older's inequality it follows that, for $1\leq p<q$,
$$\beta_{\mu,p}^n(x,r)\leq \left(\frac{\mu(B(x,r))}{r^n}\right)^{\frac1p-\frac1q}\,\beta_{\mu,q}(x,r).$$
By Theorem A, this implies that if $\mu$ is $n$-rectifiable and $1\leq p \leq2$, then
\begin{equation}\label{eqjones*p}
\int_0^1 \beta_{\mu,p}^n(x,r)^2\,\frac{dr}r<\infty \quad\mbox{ for $\mu$-a.e.\ $x\in\R^d$},
\end{equation}
and in the other direction, we also deduce that if  $0<\Theta^{n,*}(x,\mu)<\infty$ $\mu$-a.e. and \rf{eqjones*p} holds for some $p\geq2$,
then $\mu$ is $n$-rectifiable. However, from these statements one can not conclude the validity of the double implication in Theorem A for some $\beta_{\mu,p}^n$ with $p\neq2$. 
In this paper we show that indeed $p=2$ is the only case when Theorem A is valid, which may look somewhat surprising in view of the results for uniform $n$-rectifiability.
More precisely, we have:

\begin{theorem}\label{teocount2}
There exists a set $E\subset\R^2$ such that $0<\HH^1(E)<\infty$, which is purely $1$-unrectifiable, and so that,
for $1\leq p<2$,
\begin{equation}\label{eq**1}
\int_0^1\beta_{\HH^1|_E,p}^1(x,r)^2\,\frac{dr}r<\infty \quad \mbox{for $\HH^1$-a.e.\ $x\in E$.}
\end{equation}
\end{theorem}

Also:

\begin{theorem}\label{teocount3}
There exists a set $E\subset\R^2$ such that $0<\HH^1(E)<\infty$, which is $1$-rectifiable, and so that,
for all $p>2$,
\begin{equation}\label{eq**1'}
\int_0^1\beta_{\HH^1|_E,p}^1(x,r)^2\,\frac{dr}r=\infty \quad \mbox{for $\HH^1$-a.e.\ $x\in E$.}
\end{equation}
\end{theorem}

So Theorem \ref{teocount2} shows that the validity of the ``if'' direction in Theorem A requires $p\geq2$, while Theorem \ref{teocount3} shows that the other ``only if'' implication fails for $p> 2$ and thus requires $p\leq2$.

To prove Theorems  \ref{teocount2} and \ref{teocount3} we will construct some  Cantor type sets $E$ such that $\HH^1|_E$ is non-doubling. They are
obtained as limits in the Hausdorff distance of other sets $E_k$ made up of finitely many parallel segments
in the plane. It is worth comparing these sets with other counterexamples constructed in \cite{ATT}
in connection with the so-called $\alpha$ coefficients.

\vv
In this work we will also see 
that the purely $1$-unrectifiable set $E$ in Theorem \ref{teocount2} can be constructed so that,
for $1\leq p<2$,
\begin{equation}\label{eq**1''*}
\int_0^1\wt\beta_{\HH^1|_E,p}^1(x,r)^2\,\frac{dr}r<\infty \quad \mbox{for $\HH^1$-a.e.\ $x\in E$,}
\end{equation}
which, a priori, is a stronger condition than \rf{eq**1}, taking into account that $\Theta^{1,*}(x,\HH^1|_E)<\infty$
 for $\HH^1$-a.e.\ $x\in E$. This shows that Theorem A does not hold either if we replace the coefficients
 $\beta_{\mu,2}(x,r)$ by  $\wt\beta_{\mu,p}(x,r)$ for any $p\geq1$ different from $2$.
For more details, see  Theorem \ref{teocount2''} below.

\vv
In the last part of this paper we will provide a new proof of Theorem B. Indeed, we will show that this
can be derived from the results in \cite{Azzam-Tolsa-GAFA} in combination with a careful approximation argument. The techniques are quite different from the ones of Edelen, Naber and Valtorta in \cite{ENV} and use a slight variant of the corona decomposition obtained in
\cite{Azzam-Tolsa-GAFA}. On the other hand, we remark that the work \cite{ENV} contains other more quantitative results about rectifiability and $\beta_2$ numbers, apart from Theorem B. We will not consider these additional results in the current the paper.

\vv


\section{Proof of Theorem \ref{teocount2}}\label{sec2}

To shorten notation we will write $\beta_{\mu,p}(x,r)$ instead of $\beta_{\mu,p}^1(x,r)$.

Given a closed segment $I$ contained in a horizontal line in $\R^2$ and two constants $h\geq0$ and $0<a<1$ and
an integer $n\geq2$,
we denote by $I(h,a,n)$ the set made up of $n$ closed segments $J_1,\ldots, J_n$, of equal length, contained in the parallel segment $I+h\,e_2$ (where $e_2=(0,1)$), with $\sum_{i=1}^n\HH^1(J_i)=a\,\HH^1(I)$, and so that the left endpoint of
$J_1$ coincides with the left endpoint of $I+h\,e_2$, the right endpoint of
$J_n$ coincides with the right endpoint of $I+h\,e_2$, and $\dist(J_i,J_{i+1}) = \frac{1-a}{n-1} \,\HH^1(I)$ for all
$i=1,\ldots,n-1$.

Our set $E$ will be constructed as a limit in the Hausdorff distance of a sequence of compact sets $E_k$,
$k\geq 0$. 
We consider sequences $\{a_k\}_k$, $\{h_k\}_k$, $n_k$, so that $0<a_k,h_k<1$, $n_k>2$. Both $\{a_k\}_k$ and
 $\{h_k\}_k$ converge to $0$, while $n_k$ tends to $\infty$ very quickly.
Each set $E_k$, $k\geq0$, is of the form
$$E_k= \bigcup_{i=1}^{m_k} J_i^k,$$
where  $J_i^k$, $i=1,\ldots,m_k$ is a family of horizontal segments in $\R^2$ (which may be contained in different lines and
may have different lengths). The sets $E_k$ are constructed inductively. We let $E_0=[0,1]\times \{0\}$, and  we construct $E_{k+1}$ from $E_k$ as follows. We denote
$$E_{k+1}^d = \bigcup_{i=1}^{m_k}  J_i^k(0,1-a_{k+1},n_{k+1}), \quad E_{k+1}^u=\bigcup_{i=1}^{m_k}  J_i^k(h_{k+1},a_{k+1},n_{k+1})
,$$
where $J_i^k(h,a,n)$ is the set associated to the segment $J_i^k$ with parameters $h,a,n$ which was defined in the previous paragraph.
Then we set
$$E_{k+1} = E_{k+1}^d \cup E_{k+1}^u$$
(the superindices $d$ and $u$ stand for ``down'' and ``up''). See Figure \ref{figcantor}.
Observe that 
$$E_{k+1}^d\subset E_k\quad \text{and}\quad
E_{k+1}^u\subset E_k + h_{k+1}\,e_2.$$
 Also,
$$\HH^1(E_{k+1}^d) = (1-a_{k+1})\, \HH^1(E_{k})\quad \text{and}\quad \HH^1(E_{k+1}^u) = a_{k+1}\, \HH^1(E_{k}),$$
since $$\HH^1(J_i^k(0,1-a_{k+1},n_{k+1})) = (1-a_{k+1})\,\HH^1(J_i^k)\!\!\quad \text{and}\!\!\quad\HH^1(J_i^k(h_{k+1},a_{k+1},n_{k+1})) = a_{k+1}\,\HH^1(J_i^k)$$ for each $i=1,\ldots,m_k$.
Hence $\HH^1(E_{k+1})=\HH^1(E_{k})$, because the sets $J_i^k(0,1-a_{k+1},n_{k+1})$, $J_{i'}^k(h_{k+1},a_{k+1},n_{k+1})$, are pairwise disjoint (assuming $h_{k+1}$ to be small enough).
\begin{figure}[h]
\begin{center}
\quad
\hspace{5mm}
\begin{minipage}{0.4\textwidth}
\psset{unit=0.65cm}

\begin{pspicture}(11,2.5)
\rput{0}(-0.3,2.2){$E_1$}
  
  \psline[linewidth=1pt](0,0)(2,0)
  \psline[linewidth=1pt](3,0)(5,0)
  \psline[linewidth=1pt](6,0)(8,0)
  
  \psline[linewidth=1pt](0,1.5)(0.67,1.5)
  \psline[linewidth=1pt](3.67,1.5)(4.33,1.5)
  \psline[linewidth=1pt](7.33,1.5)(8,1.5)

 \end{pspicture}

\end{minipage}
\qquad\qquad\quad
\begin{minipage}{0.4\textwidth}
\psset{unit=0.65cm}
\begin{pspicture}(11,2.5)

\rput{0}(-0.3,2.2){$E_2$}
\psset{unit=0.1181818cm}  
  \psline[linewidth=1pt](0,0)(2,0)
  \psline[linewidth=1pt](3,0)(5,0)
  \psline[linewidth=1pt](6,0)(8,0)
  \psline[linewidth=1pt](9,0)(11,0)
  
  \psline[linewidth=1pt](0,1)(0.75,1)
  \psline[linewidth=1pt](3.42,1)(4.17,1)
  \psline[linewidth=1pt](6.83,1)(7.58,1)
  \psline[linewidth=1pt](10.25,1)(11,1)

\psset{unit=0.65cm}    
\rput(3,0){
  \psset{unit=0.1181818cm}  
  \psline[linewidth=1pt](0,0)(2,0)
  \psline[linewidth=1pt](3,0)(5,0)
  \psline[linewidth=1pt](6,0)(8,0)
  \psline[linewidth=1pt](9,0)(11,0)
  
  \psline[linewidth=1pt](0,1)(0.75,1)
  \psline[linewidth=1pt](3.42,1)(4.17,1)
  \psline[linewidth=1pt](6.83,1)(7.58,1)
  \psline[linewidth=1pt](10.25,1)(11,1)
 }

\psset{unit=0.65cm}    
\rput(6,0){
  \psset{unit=0.1181818cm}  
  \psline[linewidth=1pt](0,0)(2,0)
  \psline[linewidth=1pt](3,0)(5,0)
  \psline[linewidth=1pt](6,0)(8,0)
  \psline[linewidth=1pt](9,0)(11,0)
  
  \psline[linewidth=1pt](0,1)(0.75,1)
  \psline[linewidth=1pt](3.42,1)(4.17,1)
  \psline[linewidth=1pt](6.83,1)(7.58,1)
  \psline[linewidth=1pt](10.25,1)(11,1)
 }

\psset{unit=0.65cm}    
\rput(0,1.5){
  \psset{unit=0.0393939cm}  
  \psline[linewidth=1pt](0,0)(2,0)
  \psline[linewidth=1pt](3,0)(5,0)
  \psline[linewidth=1pt](6,0)(8,0)
  \psline[linewidth=1pt](9,0)(11,0)
  
  \psline[linewidth=1pt](0,1)(0.75,1)
  \psline[linewidth=1pt](3.42,1)(4.17,1)
  \psline[linewidth=1pt](6.83,1)(7.58,1)
  \psline[linewidth=1pt](10.25,1)(11,1)
 }

\psset{unit=0.65cm}    
\rput(3.67,1.5){
  \psset{unit=0.0393939cm}  
  \psline[linewidth=1pt](0,0)(2,0)
  \psline[linewidth=1pt](3,0)(5,0)
  \psline[linewidth=1pt](6,0)(8,0)
  \psline[linewidth=1pt](9,0)(11,0)
  
  \psline[linewidth=1pt](0,1)(0.75,1)
  \psline[linewidth=1pt](3.42,1)(4.17,1)
  \psline[linewidth=1pt](6.83,1)(7.58,1)
  \psline[linewidth=1pt](10.25,1)(11,1)
 }

\psset{unit=0.65cm}    
\rput(7.33,1.5){
  \psset{unit=0.0393939cm}  
  \psline[linewidth=1pt](0,0)(2,0)
  \psline[linewidth=1pt](3,0)(5,0)
  \psline[linewidth=1pt](6,0)(8,0)
  \psline[linewidth=1.25pt](9,0)(11,0)
  
  \psline[linewidth=1pt](0,1)(0.75,1)
  \psline[linewidth=1pt](3.42,1)(4.17,1)
  \psline[linewidth=1pt](6.83,1)(7.58,1)
  \psline[linewidth=1pt](10.25,1)(11,1)
 }
 \end{pspicture}

 \end{minipage}
\end{center}

\vspace{4mm}

\caption{The generations $E_1$ and $E_2$ of the Cantor set $E$, with $n_1=3$ and $n_2=4$.}
 \label{figcantor}
\end{figure}
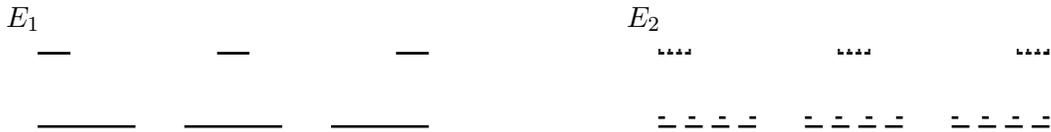

Later we will choose $\{a_k\}_k$ so that $\sum_ka_k^{2/p}<\infty$ but $\sum_ka_k =\infty$. On the other hand, we will take $h_k$ so that $\{h_k\}_k$ converges to $0$ much faster than $\{a_k\}_k$. Further, for convenience we will choose $n_k$ such that $n_k\approx 1/h_k^2$ (for example, we may take
$n_k$ as the smallest integer larger than $1/h_k^2$).
 We also assume that 
\begin{equation}\label{eqhk+1}
h_{k+1}\leq 2^{-2k-5}\,\min\Bigl(h_k,\,\min_{1\leq i \leq m_k}\HH^1(J_i^{k})\Bigr).
\end{equation}
In particular, the condition $h_{k+1}\leq 2^{-2k-5} h_k$ guaranties that $h_{k+1}$ is much smaller than 
the minimal distance among the different horizontal lines that intersect
$\supp\mu_k$ (which equals $h_k$).

We denote $\mu_{k} = \HH^1|_{E_k}$.
Next we will estimate $\beta_{\mu_{k+1},p}(x,r)$ for $x\in\supp\mu_{k+1}$ in terms of $\beta_{\mu_k,p}(x',r+c_1h_{k+1})$,
where $x'$ is the nearest point to $x$ from $\supp\mu_k$ and $c_1$ is some universal constant. That is, by construction, $x'=x$ if 
$x\in E_{k+1}^d$ and
$x'=x-h_{k+1}e_2$ if $x\in E_{k+1}^u$. 
Note first that, for our fixed point $x\in E_k$,
$$\beta_{\mu_{k+1},p}(x,r)=0\quad \mbox{ if  $0<r\leq h_{k+1}$.}$$
In the case $h_{k+1}< r\leq h_k/2$, $B(x,r)$ 
only intersects either one or two lines  from the
family of all lines which contain some segment $J_i^{k+1}$, $i=1,\ldots,m_{k+1}$.  
If it only intersects one line, then $\beta_{\mu_{k+1},p}(x,r)=0$. Otherwise, let us call $L_d$ and $L_u$ the two lines 
which contain some segment $J_i^{k+1}$, $i=1,\ldots,m_{k+1}$ and intersect $B(x,r)$, so that
$L_d$ contains segments which are contained in $E_{k+1}^d$, and $L_u$, segments from $E_{k+1}^u$.
 Further, the distance between $L_d$ and $L_u$ is $h_{k+1}$.
Then we have
\begin{align*}
\beta_{\mu_{k+1},p}(x,r)^p & \leq \frac1{r}\int_{B(x,r)}\left(\frac{\dist(y,L_d)}r\right)^p\,d\mu_{k+1}(y)\\
& = \frac1{r}\int_{B(x,r)\cap L_u}\left(\frac{\dist(y,L_d)}r\right)^p\,d\mu_{k+1}(y)
= \frac{h_{k+1}^p}{r^{p+1}}\,\mu_{k+1}(L_u\cap B(x,r)).
\end{align*}
By construction, it is easy to check that
\begin{equation}\label{eqaux95}
\mu_{k+1}(L_u\cap B(x,r))\lesssim a_{k+1} \,r
\end{equation}
To this end, notice that $n_{k+1}\approx h_{k+1}^{-2}$ and thus $$\HH^1(J^{k+1}_i)
\leq \frac1{n_{k+1}}\,\max_j \HH^1(J^{k}_j)\leq \frac1{n_{k+1}} \approx h_{k+1}^2\ll h_{k+1},$$
 for $k$ big
enough.
From \rf{eqaux95} we derive
$$\beta_{\mu_{k+1},p}(x,r) \lesssim a_{k+1}^{1/p}\,\frac{h_{k+1}}{r}.
$$
Therefore,
\begin{equation}\label{eq**5'}
\int_0^{h_k/2}
\beta_{\mu_{k+1},p}(x,r)^2\,\frac{dr}r \lesssim \int_{h_{k+1}}^{h_k/2} a_{k+1}^{2/p}\,\frac{h_{k+1}^2}{r^2}
\,\frac{dr}r 
\lesssim a_{k+1}^{2/p}\quad\mbox{ for all $x\in\supp\mu_{k+1}$.}
\end{equation}

To deal with the case $r>h_k/2$ we claim that, if  $h_{k+1}$ is small enough, then
\begin{equation}\label{eq*50'}
\beta_{\mu_{k+1},p}(x,r)^p\leq \beta_{\mu_{k},p}(x',r+c_1h_{k+1})^p
+ C\,\frac{h_{k+1}}r,
\end{equation}
for some universal constants $c_1,C$ to be fixed below.
We defer the details to the end of the proof.
Gathering the previous estimates, for any $0<\ve_{k+1}<1/2$, we obtain
\begin{align}\label{eq*5500}
\int_{h_k/2}^\infty
\beta_{\mu_{k+1},p}&(x,r)^2\,\frac{dr}r \\
& \leq 
(1+\ve_{k+1})\int_{h_{k}/2}^\infty
\beta_{\mu_{k},p}(x',r+c_1h_{k+1})^2\,\frac{dr}r + C\,\frac1{\ve_{k+1}} \int_{h_k/2}^\infty\frac{h_{k+1}^{2/p}}{r^{2/p}}\,\frac{dr}r\nonumber\\
& \leq (1+\ve_{k+1})\,\frac{\frac12h_k}{\frac12 h_k -c_1h_{k+1}}\int_0^\infty
\beta_{\mu_{k},p}(x',r)^2\,\frac{dr}r + C\,\frac{h_{k+1}^{2/p}}{\ve_{k+1}\,h_k^{2/p}},
\nonumber
\end{align}
by a change of variables in the last inequality.
Together with \rf{eq**5'}, and using that
$$\frac{\frac12h_k}{\frac12 h_k -c_1h_{k+1}}\leq 1+ C\,\frac{h_{k+1}}{h_k} ,$$
this gives
\begin{equation}\label{eq*55'}
\int_0^\infty
\beta_{\mu_{k+1},p}(x,r)^2\,\frac{dr}r 
\leq C\,a_{k+1}^{2/p}\, + (1+\ve_{k+1})
\biggl(1+ C\,\frac{h_{k+1}}{h_k}\,\biggr)\int_0^\infty
\beta_{\mu_{k},p}(x',r)^2\,\frac{dr}r + C\,\frac{h_{k+1}}{\ve_{k+1}\,h_k}.
\end{equation}
Choosing $\ve_{k+1}=2^{-k}$ and since, by \rf{eqhk+1},
\begin{equation*}
\frac{h_{k+1}}{h_k}\leq 2^{-2k}
\end{equation*}
iterating the estimate \rf{eq*55'}, it follows that
$$\int_0^\infty
\beta_{\mu_{k+1},p}(x,r)^2\,\frac{dr}r \lesssim \sum_{j=1}^{k+1} a_j^{2/p} +\sum_{j=1}^k
\frac{2^j\,h_{j+1}}{h_j} \lesssim 1+ \sum_{j=1}^{k+1} a_j^{2/p}.$$
Since this is uniform on $k$, taking a weak limit and denoting by $\mu$ the corresponding weak limit, we derive
\begin{equation}\label{eqbetap10}
\int_0^\infty
\beta_{\mu,p}(x,r)^2\,\frac{dr}r \lesssim 1 + \sum_{j\geq1} a_j^{2/p} .
\end{equation}

Consider now a sequence $\{a_j\}_j$ such that $\sum_{j\geq1} a_j^{2/p}<\infty$ but so that $\sum_j a_j=\infty$,
such as, for example, $a_j = 1/(2j)$ (recall that $1\leq p<2$).
It is easy to check that $\mu =g\,\HH^1|_E$ for some function $g\approx1$, and so $0<\HH^1(E)<\infty$.
We also postpone the detailed arguments to the end of this section. Thus the condition
\rf{eqbetap10} yields
$$\int_0^\infty
\beta_{\HH^1|_E,p}(x,r)^2\,\frac{dr}r \lesssim 1 + \sum_{j\geq1} a_j^{2/p} <\infty.
$$

It remains to prove that $E=\supp\mu$ is purely unrectifiable.
This is a consequence of the fact that $\sum_j a_j=\infty$. Indeed, 
given $x\in E$, write $x\in U_k$ if the closest point to $x$ from $E_k$ belongs to $E_k^u$, and write 
$x\in D_k$ otherwise. By the second Borel-Cantelli lemma the condition $\sum_j a_j=\infty$ implies
that $x\in U_k$ for infinitely many $k$'s. Note now that, by construction, if $x\in U_k$, then
there exists some segment $J_i^k\subset E_k^u$ such that 
$$\dist(x,J_i^k)\leq \sum_{j\geq k+1} h_j \leq \frac1{10}\,h_k,$$
because of the quick decay of $\{h _k\}$. Then, for $r=h_k/2$ and $k$ big enough, we have
$$\HH^1(B(x,r)\cap E) \leq C\, \mu(B(x,r)\cap E)\leq C\,\mu_k(B(x,1.1r)\cap E)\lesssim C\,a_k\,r.$$
Therefore, if $x\in U_k$ for infinitely many $k$'s, since $a_k\to0$ as $k\to\infty$, then
$$\liminf_{r\to0}\frac{\HH^1(E\cap B(x,r))}{2r}=0.$$
As this happens for $\HH^1$-a.e. $x\in E$, it turns out that $E$ is purely unrectifiable (see
Theorem 17.6 in \cite{Mattila-book}, for example).

\vvv

\vv
\noi{\bf Proof of \rf{eq*50'}}.
Split each segment $J_i^k$, $i=1,\ldots,m_k$, into $n_{k+1}$ segments with disjoint interiors and equal length, and 
denote by $\{I_j^{k+1}\}_{1\leq j\leq m_{k+1}/2}$ the resulting family of segments obtained from such splitting. Let $I_j^{k+1,l}$ be the leftmost closed sub-segment of $I_j^{k+1}$ of length $(1-a_{k+1})\HH^1(I_j^{k+1})$ and let $I_j^{k+1,r}$ be the rightmost half open-closed sub-segment of $I_j^{k+1}$ of length $a_{k+1}\HH^1(I_j^{k+1})$, so that $I_j^{k+1}=I_j^{k+1,l}\cup I_j^{k+1,r}$ and the union is disjoint.

Suppose that the family of segments $\{J_j^{k+1}\}_{1\leq j\leq m_{k+1}}$ is labeled so that the indices
$j=1,\ldots,m_{k+1}/2$ correspond to the subfamily of the segments $J_j^{k+1}$ which are contained in 
$E_{k+1}^d$, and assume also that the labeling is so that, for each $1\leq j\leq m_{k+1}/2$,
$J_j^{k+1}$ is the closest segment (in Hausdorff distance) from $\{J_{j'}^{k+1}\}_{1\leq j'\leq m_{k+1}/2}$
to $I_j^{k+1}$. Also, for $j=1,\ldots,m_{k+1}/2$, given some segment $J_j^{k+1}\subset J_i^k$,
denote by  $J_j^{k+1,u}$ a segment from the family $\{J_{j'}^{k+1}\}_{m_{k+1}/2\leq j'\leq m_{k+1}}$
which is contained in $h_{k+1}e_2+J_i^k$ and is at a distance at most $c\,h_{k+1}$ from  $J_j^{k+1}$,
where $c$ is some absolute constant. 
By our geometric construction, it is easy to check that such choice can be done so that
the segments from $\{J_{j'}^{k+1,u}\}_{1\leq j'\leq m_{k+1}/2}$ are pairwise different (i.e. the correspondence 
 $J_j^{k+1}\mapsto  J_j^{k+1,u}$ is one to one).
 
Now we consider the map $T^{k+1}:\supp\mu_{k}\to\supp\mu_{k+1}$ defined as follows. For each $j=1,\ldots,m_{k+1}/2$ we denote by $T_j^{k+1,l}$ the translation such that $T_j^{k+1,l}(I_j^{k+1,l})=J_j^{k+1}$ and  
by $T_j^{k+1,r}$ the translation such that $T_j^{k+1,r}(I_j^{k+1,r})=(J_j^{k+1,u})$. Now, for each  $j=1,\ldots,m_{k+1}/2$, we set $T^{k+1}(x) = T_j^{k+1,l}(x)$ if $x\in 
I_j^{k+1,l}$, and $T^{k+1}(x) = T_j^{k+1,r}(x)$ if $x\in 
I_j^{k+1,r}$.
Then it is easy to check that, for all $x\in\supp\mu_{k+1}$,
\begin{equation}\label{eqtk+1}
|x-T^{k+1}(x)|\leq C\,h_{k+1},
\end{equation}
and further
\begin{equation}\label{eqtk+2}
T^{k+1}\#\mu_{k} = \mu_{k+1}.
\end{equation}

To estimate $\beta_{\mu_{k+1},p}(x,r)$ for $r\geq h_k/2$, let $L$ be some line minimizing 
$\beta_{\mu_{k},p}(x',r+c_1h_{k+1})$ for some constant $c_1\approx1$ to be fixed below.
Then we have
\begin{align}\label{eqali99'}
\beta_{\mu_{k+1},p}(x,r)^p & \leq \int_{B(x,r)} \left(\frac{\dist(y,L)}r\right)^p\,
\frac{d(T^{k+1}\#\mu_{k})(y)}{r}\\
& = \int_{(T^{k+1})^{-1}(B(x,r))} \left(\frac{\dist(T^{k+1}(y),L)}r\right)^p\,
\frac{d\mu_k(y)}{r}.\nonumber
\end{align}
To deal with the last integral above, we take into account that
$$
\left|\left(\frac{\dist(y,L)}r\right)^p-\left(\frac{\dist(T^{k+1}(y),L)}r\right)^p\right|
\lesssim \frac{|y-T^{k+1}(y)|}r\lesssim\frac{h_{k+1}}r.
$$
Using also that  $\mu_k((T^{k+1})^{-1}(B(x,r)))\leq \mu_k(B(x',r+c_1h_{k+1}))\leq c_2\,r$ for some universal constants $c_1$ and $c_2$ (see Remark \ref{remucreix1} for more details), we obtain
$$\beta_{\mu_{k+1},p}(x,r)^p  \leq \int_{B(x',r+c_1h_k)} \left(\frac{\dist(y,L)}r\right)^p\,
\frac{d\mu_{k}(y)}{r} + c\,\frac{h_{k+1}}r = \beta_{\mu_{k},p}(x',r+c_1h_k)^p+ c\,\frac{h_{k+1}}r,
$$
which proves \rf{eq*50'}.

\vvv

\vv
\noi{\bf Proof of $\mu=g\,\HH^n|_E$ for some $g\approx1$}. First we show that  
$\HH^n|_E\leq \mu|_E$. To this end we consider a family of ``dyadic cubes'' 
$$\DD_E=\{Q^k_j:k\geq 0, 1\leq j\leq m_k\}$$
defined as follows.
For $k\geq 0$ and $1\leq j\leq m_k$, consider a segment $J_j^k$ from the construction of $E$. Then denote by $R^k_j$ the closed rectangle whose bases are $J_j^k$ and $2h_{k+1}e_2+ J_j^k$, and set
$$Q_j^k = R_j^k \cap E.$$
Alternatively, one can think that $Q_j^k$ is the limit in the Hausdorff distance of the set
$$T^{k+i}\circ T^{k+i-1}\circ\ldots\circ T^k(J_j^k)$$
as $i\to\infty$. Observe that, by \rf{eqhk+1}
$$\diam Q_j^k\leq \diam R_j^k\leq 
\HH^1(J_j^k) + 2h_{k+1}\leq (1+2^{-2k-4})\,\HH^1(J_j^k).
$$
Thus, by the covering $Q^k_j = \bigcup_{i:Q^{k+h}_{i}\subset Q^k_j} Q^{k+h}_{i}$ and setting $\ve_h=\max_{i:Q^{k+h}_{i}\subset Q^k_j}\diam(Q^{k+h}_{i})$, it follows that
$$\HH^1_{\ve_h}(Q^k_j) \leq \!\sum_{i:Q^{k+h}_{i}\subset Q^k_j} \!\!\diam(Q^{k+h}_{i})
\leq (1+2^{-2k-2h-4})\!\sum_{i:Q^{k+h}_{i}\subset Q^k_j}\!\!\HH^1(J_i^{k+h}) \leq (1+2^{-2k-2h})\,\HH^1(J_j^k).
$$
So, letting $h\to\infty$,
$\HH^1(Q^k_j) \leq \HH^1(J_j^k) = \mu(Q^k_j).$
Since any relatively open subset $G\subset E$ can be split into a countable disjoint union of cubes from $\DD_E$, one deduces that $\HH^1(G)\leq \mu(G)$. By the regularity of $\HH^1|_E$ and $\mu$, this implies that 
$\HH^1|_E\leq \mu$.

To show that $\mu\leq g\,\HH^n|_E$
for some $g\lesssim1$, it is enough to prove that 
\begin{equation}\label{eqa01}
\mu(A)\leq C\,\diam(A)\quad \mbox{ for any Borel set $A\subset\R^2$.}
\end{equation}
 Indeed, given any subset $F\subset E$, any arbitrary covering $F\subset\bigcup_i A_i$ satisfies
$$\mu(F)\leq \sum_i\mu(A_i)\leq C\sum_i\diam(A_i),$$
which implies that $\mu(F)\leq C\,\HH^1(F)$, by the definition of $\HH^1$.

To prove \rf{eqa01}, let $k$ be such that
$\frac12 h_{k+1}\leq \diam A<\frac12h_k$. Denote by $\{L^k_i\}_{1\leq i \leq 2^k}$ the family of lines
which contain some segment from the family $\{J_j^k\}_{1\leq j\leq m_k}$, and recall
that the distance between two different lines $L_i^k$, $L_{i'}^k$ is at least $h_k$.
Also,  it is easy to check that, by construction, $\mu_k|_{L_i^k} \leq \HH^1|_{L_i^k}$ for each $i$.
Therefore, any set $A'$ intersecting at most one of such lines satisfies $\mu_k(A')\leq\diam(A')$.
Recall now, that for any $j$,
$$\mu_{j+1}(A) = T^{j+1}\#\mu_j(A) = \mu_j((T^{j+1})^{-1}(A)) \leq \mu_j(U_{C\,h_{j+1}}(A)),$$
taking into account \rf{eqtk+2} and \rf{eqtk+1}, and denoting by $U_{t}(A)$ the $t$-neighborhood of $A$.
Iterating the preceding estimate, for $j\geq k$ we get
$$\mu_j(A)\leq \mu_{j-1}(U_{C\,h_{j}}(A))\leq \ldots \leq\mu_k(U_{C\,h_{j}+\ldots +C\,h_{k+1} }(A))
\leq \mu_k(U_{C'\,h_{k+1} }(A)),$$
taking also into account the fast decay of the sequence $\{h_k\}_k$, by \rf{eqhk+1}.
Since the set $A':=U_{C'\,h_{k+1} }(A)$ intersects at most one line $L^k_i$ (assuming $k$ big enough), we deduce that
\begin{equation}\label{eqmujr}
\mu_j(A)\leq \diam(U_{C'\,h_{k+1} }(A)) \leq \diam(A)+2C'\,h_{k+1}\leq C\,\diam(A)
\end{equation}
for all $j\geq k$.
Letting $j\to\infty$, we infer that any set $A\subset\R^2$ satisfies
$\mu(A)\leq C\,\diam(A)$ as wished.\footnote{By a more careful argument, one can show that 
$\mu(A)\leq \diam(A)$ for any Borel set $A\subset\R^2$, which implies that $\mu=\HH^1|_E$.}

\vv
\begin{remark}\label{remucreix1}
The arguments above also show that
\begin{equation}\label{eqmujr2}
\mu_j(A)\leq C\,\diam(A)\quad\mbox{ for any Borel set $A\subset\R^2$.}
\end{equation}
Indeed, \rf{eqmujr} shows that this holds if $\frac12 h_{k+1}\leq \diam A<\frac12h_k$ for
some $k\leq j$. In the case $\diam A<\frac12 h_{j+1}$, then $A$ intersect at most one line
$L_i^j$ and so \rf{eqmujr2} also holds.
\end{remark}

\vv


\section{The counterexample involving the $\wt\beta_p$ coefficients}\label{sec3}

In this section we will prove the following.

\begin{theorem}\label{teocount2''}
There exists a set $E\subset\R^2$ such that $0<\HH^1(E)<\infty$, which is purely $1$-unrectifiable, and so that,
for $1\leq p<2$,
\begin{equation}\label{eq**1''}
\int_0^1\wt\beta_{\HH^1|_E,p}^1(x,r)^2\,\frac{dr}r<\infty \quad \mbox{for $\HH^1$-a.e.\ $x\in E$.}
\end{equation}
\end{theorem}

Let us remark that, for sets $E\subset\R^2$ such that $0<\HH^1(E)<\infty$,  the condition \rf{eq**1''} implies
\rf{eq**1}, and thus Theorem \ref{teocount2} is implied by the theorem above.
However, we have preferred to prove first Theorem \ref{teocount2} separately because its proof is a little more transparent and
less technical than the one of Theorem \ref{teocount2''}

\begin{proof}
To shorten notation we write
 $\wt\beta_{\mu,p}(x,r)$ instead of $\wt\beta_{\mu,p}^1(x,r)$.
 
We consider exactly the same set $E$ constructed in the previous section, and we use the same notation.
We also choose $a_k=1/(2k)$ and $h_k$ as in \rf{eqhk+1}, and also $n_k\approx 1/h_k^2$.
We have already shown that $0<\HH^1(E)<\infty$ and that $E$ is purely $1$-rectifiable, and thus 
we just have to show that
\rf{eq**1''} holds for $1\leq p<2$ if $h_k$ decreases fast enough as $k\to\infty$ (besides satisfying
\rf{eqhk+1}).
Further, we may assume that $1<p<2$ because $\wt\beta_{\HH^1|_E,1}(x,r)\lesssim
\wt\beta_{\HH^1|_E,p}(x,r)$ for such $p$'s.

To prove \rf{eq**1''} we will follow some arguments quite similar to the ones 
 in the preceding section.
Clearly, for any $x\in E_k$,
$$\wt\beta_{\mu_{k+1},p}(x,r)=0\quad \mbox{ if $0<r\leq h_{k+1}$.}$$
In the case $h_{k+1}< r\leq h_k/2$, $B(x,r)$ (still for $x\in E_k$) only intersects either one or two lines  from the
family of all lines which contain some segment $J_i^{k+1}$, $i=1,\ldots,m_{k+1}$.  
If it only intersects one line, then $\wt\beta_{\mu_{k+1},p}(x,r)=0$. Otherwise, let we call $L_d$ and $L_u$ the two lines 
which contain some segment $J_i^{k+1}$, $i=1,\ldots,m_{k+1}$ and intersect $B(x,r)$, so that
$L_d$ contains segments which are contained in $E_{k+1}^d$, and $L_u$, segments from $E_{k+1}^u$.
Then we have
\begin{align*}
\wt\beta_{\mu_{k+1},p}(x,r)^p & \leq \frac1{\mu_{k+1}(B(x,r))}\int_{B(x,r)}\left(\frac{\dist(y,L_d)}r\right)^p\,d\mu_{k+1}(y)\\
& 
= \frac{h_{k+1}^p}{r^{p}}\,\frac{\mu_{k+1}(L_u\cap B(x,r))}{\mu_{k+1}(B(x,r))}.
\end{align*}
Then, from \rf{eqaux95} it follows that
\begin{equation}\label{eqbeta94}
\wt\beta_{\mu_{k+1},p}(x,r) \lesssim a_{k+1}^{1/p}\,\frac{h_{k+1}}{r}\,\left(\frac r{{\mu_{k+1}(B(x,r))}}\right)^{1/p}.
\end{equation}

Now we need to estimate $\mu_{k+1}(B(x,r))$ from below, still assuming that $h_{k+1}< r\leq h_k/2$. It is easy to check that
\begin{equation}\label{eqsk1001}
\mu_{k+1}(B(x,r))\gtrsim \HH^1(B(x,r)\cap L_d) \approx r\quad\mbox{ if  $x\in L_d$.}
\end{equation}
In the  case $x\in L_u$ we write
$$
\mu_{k+1}(B(x,r))\geq \mu_{k+1}(B(x,r)\cap L_d).
$$
Observe that 
$$\HH^1(B(x,r)\cap L_d) = 2\sqrt{r^2-h_{k+1}^2}.$$
Recall that, by construction, each of the segments $J_i^{k+1}$, $i=1,\ldots,m_{k+1}$, which are contained in $L_u$ is also contained in a set 
$
J_j^k(0,1-a_{k+1},n_{k+1})$ for some $j\in[1,m_k]$. 
Denote 
$$s_{k+1} = \max_i \,\dist(J_i^{k+1}, E_{k+1}\setminus  J_i^{k+1}).$$
Clearly, by construction,
\begin{equation}\label{eqsk+1}
s_{k+1}\leq \frac1{n_{k+1}-1}\,\max_i\,\HH^1(J_i^{k+1})\leq
\frac1{n_{k+1}-1}
\approx h_{k+1}^2.
\end{equation}
It follows easily that if $\HH^1(B(x,r)\cap L_d)\geq 2\,s_{k+1}$ 
(or equivalently, $r^2\geq h_{k+1}^2 + s_{k+1}^2$), then
$$\mu_{k+1}(B(x,r)\cap L_d)\approx\HH^1(B(x,r)\cap L_d) =2 \sqrt{r^2-h_{k+1}^2}.$$
Hence, for $h_{k+1} + s_{k+1}\leq r\leq h_k/2$ we have
$$\mu_{k+1}(B(x,r))\gtrsim \sqrt{r^2-h_{k+1}^2}.$$
By \rf{eqsk1001} this estimate also holds for $x\in L_d$. Together with \rf{eqbeta94}, this implies
that for all $x\in E_{k+1}$ and 
 $h_{k+1} + s_{k+1}\leq r\leq h_k/2$
we have
$$\wt\beta_{\mu_{k+1},p}(x,r)^2 \lesssim a_{k+1}^{2/p}\,\frac{h_{k+1}^2}{r^2}\,\left(\frac {r^2}{r^2-h_{k+1}^2}\right)^{1/p}\!\approx \,a_{k+1}^{2/p}\,\frac{h_{k+1}^2}{r^2}\,\left(\frac {r}{r-h_{k+1}}\right)^{1/p}.
$$

From the preceding estimate we deduce
\begin{align*}
\int_0^{h_k/2}
\wt\beta_{\mu_{k+1},p}(x,r)^2\,\frac{dr}r & \lesssim 
\int_{h_{k+1}}^{h_{k+1} + s_{k+1}} 
\,\frac{dr}r +
\int_{h_{k+1}+s_{k+1}}^{h_k/2} 
a_{k+1}^{2/p}\,\frac{h_{k+1}^2}{r^2}\,\left(\frac {r}{r-h_{k+1}}\right)^{1/p}
\,\frac{dr}r \\ 
\mbox{}\\
& =: I_1 + I_2.
\end{align*}
Note that, by \rf{eqsk+1},
$$I_1 = \log\frac{h_{k+1} + s_{k+1}}{h_{k+1}} \approx \frac{s_{k+1}}{h_{k+1}}\lesssim h_{k+1}.$$
Concerning $I_2$, we write
\begin{align*}
I_2 & \leq \int_{h_{k+1}}^{h_k/2} a_{k+1}^{2/p} 
\,\frac{h_{k+1}^2}{r^2}\,\left(\frac {r}{r-h_{k+1}}\right)^{1/p}
\,\frac{dr}r \\
& \leq\int_{h_{k+1}}^{2h_{k+1}} \cdots \;\,+ \int_{2h_{k+1}}^\infty \cdots\\
& \lesssim a_{k+1}^{2/p}\,h_{k+1}^{1/p-1} \int_{h_{k+1}}^{2h_{k+1}}\frac {1}{(r-h_{k+1})^{1/p}}\,dr + a_{k+1}^{2/p} \,h_{k+1}^2
\int_{2h_{k+1}}^\infty 
\,\frac{1}{r^2}\,
\,\frac{dr}r\\
&\approx a_{k+1}^{2/p},
\end{align*}
using the fact that  $p>1$ to estimate the first integral in the before to last line.
So we have
\begin{equation}\label{eqbeta743}
\int_0^{h_k/2}
\wt\beta_{\mu_{k+1},p}(x,r)^2\,\frac{dr}r \lesssim h_{k+1} + a_{k+1}^{2/p}\approx a_{k+1}^{2/p},
\end{equation}
assuming that $h_{k+1} \ll a_{k+1}^{2/p}$.
\vv

For $r>h_k/2$ and $x\in\supp\mu_{k+1}$ we will estimate $\wt\beta_{\mu_{k+1},p}(x,r)$ in terms of $\wt\beta_{\mu_k,p}(x',r+c_1h_{k+1})$,
where $x'$ is again the nearest point to $x$ from $\supp\mu_k$ and $c_1$ is some universal constant. 
By \rf{eq*50'}, denoting $r'= r+c_1h_{k+1}$, we have
\begin{align*}
\wt\beta_{\mu_{k+1},p}(x,r)^p & = \frac r{\mu_{k+1}(B(x,r))}\, \beta_{\mu_{k+1},p}(x,r)^p \\
& \leq \frac r{\mu_{k+1}(B(x,r))}\,\biggl(
\beta_{\mu_{k},p}(x',r')^p + C\,\frac{h_{k+1}}r\biggr)\\
& = \frac r{\mu_{k+1}(B(x,r))}\,\biggl(\frac{\mu_k(B(x',r'))}{r'}\,
\wt\beta_{\mu_{k},p}(x',r')^p + C\,\frac{h_{k+1}}r\biggr).
\end{align*}
Observe that 
\begin{equation}\label{eqsh3}
\mu_{k+1}(B(x,r)) = \mu_k((T^{k+1})^{-1}(B(x,r))) \geq  \mu_k(B(x',r-c_3h_{k+1}))\geq \mu_k(B(x',r/2))\geq c_k \,r,
\end{equation}
where $c_3$ is some universal constant and $c_k$ is some constant depending on the parameters
$k,a_1,\ldots,a_k,h_1,\ldots,h_k,n_1,\ldots,n_k$ (probably the estimate \rf{eqsh3}
can be sharpened, but this is enough for us).
So taking also into account that $r'>r$, we derive
\begin{equation}\label{eqbeta574}
\wt\beta_{\mu_{k+1},p}(x,r)^p\leq \frac{\mu_k(B(x',r'))}{\mu_{k+1}(B(x,r))}\,
\wt\beta_{\mu_{k},p}(x',r')^p + C\,c_k^{-1}\,\frac{h_{k+1}}r.
\end{equation}
We have
\begin{align*}
\left|1- \frac{\mu_k(B(x',r'))}{\mu_{k+1}(B(x,r))}\right| & = 
\frac{\bigl|\mu_k(B(x',r'))-\mu_{k+1}(B(x,r))\bigr|}{\mu_{k+1}(B(x,r))}
\\
& = 
\frac{\bigl|\mu_k(B(x',r'))-\mu_{k}((T^{k+1})^{-1}(B(x,r)))\bigr|}{\mu_{k+1}(B(x,r))}\\
&\leq \frac{\mu_k(A(x',r-c_4\,h_{k+1},r+c_4\,h_{k+1}))}{c_k\,r}
\end{align*}
for some universal constant $c_4$.
To estimate $\mu_k(A(x',r-c_4\,h_{k+1},r+c_4\,h_{k+1}))$ we take into account that 
$E_k$ is contained in the union of $2^k$ horizontal lines, and we use the brutal inequality
$$\mu_k(A(x',r-c_4\,h_{k+1},r+c_4\,h_{k+1}))\leq 2^k\,\sup_L\,\HH^1(A(x',r-c_4\,h_{k+1},r+c_4\,h_{k+1})\cap
L),$$
where the supremum is taken over all lines $L$. One can easily to check that
$$\sup_L\,\HH^1(A(x',r-c_4\,h_{k+1},r+c_4\,h_{k+1})\cap
L) = \sqrt{( r+c_4\,h_{k+1})^2 - ( r-c_4\,h_{k+1})^2} = \sqrt{2c_4\,r\,h_{k+1}}.$$ 
Therefore,
$$\left|1- \frac{\mu_k(B(x',r'))}{\mu_{k+1}(B(x,r))}\right|\leq 
\frac{2^k\,\sqrt{2c_4\,r\,h_{k+1}}}{c_k\,r} =: C_k\,\left(\frac{h_{k+1}}r\right)^{1/2}.$$
Together with \rf{eqbeta574}, this gives
\begin{align*}
\wt\beta_{\mu_{k+1},p}(x,r)^p &\leq \left(1+C_k\,\frac{h_{k+1}^{1/2}}{r^{1/2}}\right)\,
\wt\beta_{\mu_{k},p}(x',r')^p + C\,c_k^{-1}\,\frac{h_{k+1}}r\\
&\leq 
\wt\beta_{\mu_{k},p}(x',r')^p + C\,C_k\,\left(\frac{h_{k+1}}r\right)^{1/2} + C\,c_k^{-1}\,\frac{h_{k+1}}r\\
&\leq 
\wt\beta_{\mu_{k},p}(x',r')^p + \wt C_k\,\left(\frac{h_{k+1}}r\right)^{1/2},
\end{align*}
which should be compared with \rf{eq*50'}.
Arguing now as in  \rf{eq*5500}, \rf{eq*55'}, \rf{eqbetap10}, 
if we take $h_{k+1}$ small enough (depending on $\wt C_k$), by iterating the estimate above, we obtain
$$
\int_0^\infty
\wt\beta_{\mu,p}(x,r)^2\,\frac{dr}r \lesssim 1 + \sum_{j\geq1} a_j^{2/p}<\infty,
$$
recalling that $a_j=1/(2j)$.
We leave the details for the reader.
\end{proof}

\vv


\section{Proof of Theorem \ref{teocount3}}

We consider the same construction as in Sections \ref{sec2} and \ref{sec3} for the proofs of Theorems \ref{teocount2} and \ref{teocount2''}, respectively, and we use the same notation.
However, now we choose 
$$a_j = \frac1{j\,(\log (e+j))^2}.$$
In this way, by the estimate \rf{eqbetap10} with $p=2$, for any $x\in E$,
$$\int_0^\infty
\beta_{\mu,2}(x,r)^2\,\frac{dr}r \lesssim 1 + \sum_{j\geq1} a_j<\infty,$$
and so $E$ is $n$-rectifiable, by Theorem A.

We will show now that, for all $p>2$, 
$$\int_0^\infty
\beta_{\mu,p}(x,r)^2\,\frac{dr}r=\infty\quad \mbox{ for all $x\in E$.}$$
To this end, consider a ball $B(x,r)$, with $x\in E$ and $r$ such that $2h_{k}\leq r\leq 4h_k$.
Let  $x'\in E_k$ be the closest point to $x$ from $E_k$.
By construction $B(x',\frac32r)$ intersects two lines $L^d$ and $L^u$ which contain segments $J_i^k$ from
$E_k^d$ and $E_k^u$ respectively, so that moreover,
$$\mu_k(L^d\cap B(x',\tfrac32r))\gtrsim (1-a_k)\,r\approx r,$$
and 
$$\mu_k(L^u\cap B(x',\tfrac32r))\gtrsim a_k\,r.$$
Consider an arbitrary line $L\subset\R^2$ and denote $B=B(x',\tfrac32r)$. If $\dist_H(L\cap B,L^d\cap B)\leq \frac1{100}\,h_k$, then one can easily check that
$\dist_H(L\cap B,L^u\cap B)\gtrsim h_k\approx r$, and then it easily follows that
$$\int_{B(x',\tfrac32 r)}\left(\frac{\dist(y,L)}r\right)^p\,\frac{d\mu_k(y)}{r} \gtrsim 
\frac{\mu_k(B(x',\tfrac32r\cap L^u))}r\approx a_k.$$
Since $h_{k+1}\ll h_k$, it is easy to check that in fact we also have
$$\int_{B(x',\tfrac32 r)}\left(\frac{\dist(y,L)}r\right)^p\,\frac{d\mu_j(y)}{r} \gtrsim 
\frac{\mu_k(B(x',\tfrac32r\cap L^u))}r\approx a_k$$
uniformly for all $j\geq k$, with $k$ big enough.
Hence, by taking a weak limit, 
$$\int_{B(x,r)}\left(\frac{\dist(y,L)}r\right)^p\,\frac{d\mu(y)}{r} \gtrsim 
 a_k.$$
On the other hand,
if $\dist_H(L\cap B,L^d\cap B)> \frac1{100}\,h_k$, 
then it easily follows that
$$\int_{B(x',\tfrac32 r)}\left(\frac{\dist(y,L)}r\right)^p\,\frac{d\mu_k(y)}{r} \gtrsim 
\frac{\mu_k(B(x',\tfrac32r\cap L^d))}r\approx 1-a_k\approx 1.$$
Since $h_{k+1}\ll h_k$, then we also have
$$\int_{B(x',\tfrac32 r)}\left(\frac{\dist(y,L)}r\right)^p\,\frac{d\mu_j(y)}{r} \gtrsim 
1$$
uniformly for all $j\geq k$ with $k$ big enough, and then letting $j\to\infty$,
$$\int_{B(x,r)}\left(\frac{\dist(y,L)}r\right)^p\,\frac{d\mu(y)}{r} \gtrsim 1.$$

So in any case, for any line $L$ we have
$$\int_{B(x,r)}\left(\frac{\dist(y,L)}r\right)^p\,\frac{d\mu(y)}{r} \gtrsim \min(1,a_k)=a_k,$$
and thus
$$\beta_{\mu,p}(x,r) \gtrsim a_k^{1/p}\quad \mbox{ for $2h_k\leq r\leq 4h_k$.}$$
Therefore, for $p>2$,
$$\int_0^\infty
\beta_{\mu,p}(x,r)^2\,\frac{dr}r\geq\sum_{k=1}^\infty \int_{2h_k}^{4h_k} a_k^{2/p}\,\frac{dr}r
\approx \sum_{k=1}^\infty  \frac1{j^{2/p}\,(\log (e+j))^{4/p}} = \infty,$$
which concludes the proof of Theorem \ref{teocount3}.

\vv


\section{Rectifiability of measures with bounded lower density}

In this section we will deduce Theorem B of Edelen, Naber and Valtorta from the corona decomposition in \cite{Azzam-Tolsa-GAFA} and a suitable approximation argument.


\vv

\subsection{The dyadic lattice and the corona decomposition from \cite{Azzam-Tolsa-GAFA}}\label{sec41}
We recall that one of the main ingredients of the proof of Theorem A in \cite{Azzam-Tolsa-GAFA} is 
a corona decomposition in terms of  a dyadic lattice $\DD_\mu$ associated to the measure $\mu$, which we assume to be compactly supported. We have $\DD_\mu = \bigcup_{k\geq k_0} \DD_{\mu,k}$, and each family 
$\DD_{\mu,k}$ consists of a collection of Borel subsets  (or ``cubes'') of $E=\supp\mu$ which form a partition of $E$.
That is, for each $k\geq k_0$,
$$E = \bigcup_{Q\in \DD_{\mu,k}} Q,$$
and the union is disjoint.  Further,
if $k<l$, $Q\in\DD_{\mu,l}$, and $R\in\DD_{\mu,k}$, then either $Q\cap R=\varnothing$ or else $Q\subset R$.

The general position of the cubes $Q$ can be described as follows. There are constants $A_0,C_0\gg1$ so that for each $k\geq k_0$ and each cube $Q\in\DD_{\mu,k}$, there is a ball $B(Q)=B(z_Q,r(Q))$ such that
$$z_Q\in E, \qquad A_0^{-k}\leq r(Q)\leq C_0\,A_0^{-k},$$
$$E\cap B(Q)\subset Q\subset E\cap 28\,B(Q)=E \cap B(z_Q,28r(Q)),$$
and
$$\mbox{the balls $5B(Q)$, $Q\in\DD_{\mu,k}$, are disjoint.}$$
For other additional properties of this lattice (constructed by David and Mattila in \cite{David-Mattila}) see Lemma 2.1 from \cite{Azzam-Tolsa-GAFA}.

We set
$\ell(Q)= 56\,C_0\,A_0^{-k}$ and we call it the side length of $Q$. Note that 
$$\frac1{28}\,C_0^{-1}\ell(Q)\leq \diam(B(Q))\leq\ell(Q).$$
 We also denote
$B_Q=28 \,B(Q)=B(z_Q,28\,r(Q))$, so that 
$$E\cap \tfrac1{28}B_Q\subset Q\subset B_Q.$$

\vv

The corona decomposition  from \cite{Azzam-Tolsa-GAFA} is a partition of $\DD_\mu$ into tree-like families
whose family of associated roots is denoted by $\ttt_\mu$. The only properties of this corona decomposition that here we need to know here are the following:
\begin{enumerate}
\item $\ttt_\mu \subset \DD_\mu$, $E\in \ttt_\mu$, and each $R\in\ttt_\mu$ satisfies
$$\mu(2B_R)\leq C\,\mu(R).$$

\item For $R\in\ttt_\mu$, let $\TT(R)$ be the subfamily of cubes from $\DD_\mu$ which are contained in $R$
and which are not contained in any other cube from $\ttt_\mu$. Then
$$\DD_\mu =\bigcup_{R\in\ttt_\mu}\TT(R),$$
and the union is disjoint.

\item For each $R\in\ttt_\mu$ and each $Q\in\TT(R)$,
$$\Theta_\mu(2B_Q)\leq C\,\Theta_\mu(2B_R),$$
where $\Theta_\mu(B(x,r))=\frac{\mu(B(x,r))}{r^n}$.

\item If $$\mu\biggl(R\setminus \bigcup_{Q\in \TT(R)} Q\biggr)>0,$$ then $\TT(R)$ contains cubes $Q\in\DD_\mu$ of arbitrarily small side length satisfying  $\Theta_\mu(2B_R)\approx \Theta_\mu(2B_Q)$.

\item If $\mu$ satisfies the growth condition
\begin{equation}\label{eqgrowth00}
\mu(B(x,r))\leq C_*\,r^n\quad \mbox{ for all $x\in E$, $r>0$,}
\end {equation}
then $\ttt_\mu$ satisfies the packing condition
\begin{equation}\label{eqpack00}
\sum_{R\in\ttt_\mu}\Theta_\mu(2B_R)\,\mu(R)\leq C\,C_*\,\mu(R_0)+
C\,\int\!\!\int_0^\infty\beta_{\mu,2}(x,r)^2\,\frac{dr}r\,d\mu(x).
\end{equation}
\end{enumerate}

The properties above are proved in Section 5 from \cite{Azzam-Tolsa-GAFA}. We also remark that another
key property is the fact that, in  a sense, the measure $\mu$ is quite well approximated 
by $n$-dimensional Hausdorff measure in a Lipschitz $n$-dimensional manifold
at the scales and locations of
each tree $\TT(R)$. However, this property will not be used here and so we skip the details.

\vv
In \cite{Azzam-Tolsa-GAFA} the growth condition \rf{eqgrowth00} is only used to prove the packing condition
\rf{eqpack00}. Indeed this is not used in connection with the other 
properties of the corona decomposition listed above.

We claim now that the packing condition \rf{eqpack00} also holds if instead of \rf{eqgrowth00} we just assume that there exists some $r_0>0$ such that
\begin{equation}\label{eqgrowth00'}
\mu(B(x,r))\leq C_*\,r^n\quad \mbox{ for all $x\in E$, $0< r\leq r_0$,}
\end {equation}
with the constants in \rf{eqpack00} independent of $r_0$.
The only required modifications are located in the proof of Lemma 5.5 from \cite{Azzam-Tolsa-GAFA}.
They are quite minor and we just sketch them, and advise the reader to have \cite{Azzam-Tolsa-GAFA} at hand
to follow the details:
\begin{itemize}
\item Equation (5.9) from \cite{Azzam-Tolsa-GAFA} is still valid under the assumption \rf{eqgrowth00'},
because for $k$ big enough, $\Theta_\mu(2B_R)\leq C_*$ for all $R\in\ttt_k$.
\item
To estimate the first sum on the right hand side of (5.11) from \cite{Azzam-Tolsa-GAFA} we take into account that if
$$\mu\biggl(R\setminus \bigcup_{Q\in \nex(R)} Q\biggr)>0,$$ then  $\Theta(2B_R)\approx \Theta(2B_Q)\lesssim C_*$ for infinitely many $Q\in\TT(R)$, by the above property (4) of the corona decomposition.

\item Also, $S_2=0$ because we are taking $F=E$ and so $\BZ=\varnothing$ in \cite[Section 5]{Azzam-Tolsa-GAFA}.
\end{itemize}

\vv
\subsection{Preliminaries for the proof of Theorem B}

To prove Theorem B it is enough to show that any subset $F\subset E$ with $\mu(F)>0$ contains another subset
$F'\subset F$ with $\mu(F')>0$ which is $n$-rectifiable. Having this in mind, by standard methods, it is easy to check
 that we can assume that, for some constants $C_*$ and $C_1$,
\begin{equation}\label{eqc*'}
\liminf_{r\to0} \frac{\mu(B(x,r))}{r^n}\leq C_*\quad \mbox{ for all $x\in E$},
\end{equation}
and
\begin{equation}\label{eqjones*'}
\int_0^1 \beta_{\mu,2}(x,r)^2\,\frac{dr}r\leq C_{1} \quad\mbox{ for all $x\in E$.}
\end{equation}

\vv
We need the following auxiliary result.

\begin{lemma}\label{lemdb*}
Let $\Lambda>2$.
Under the assumption \rf{eqc*'}, for $\mu$-a.e.\ $x\in E$ there exists a sequence of radii $r_k\to0$
such that
\begin{equation}\label{eqrk0'}
\mu(B(x,\Lambda r_k)) \leq 2\Lambda^{d}\,\mu(B(x,r_k))
\quad \mbox{ and }\quad \mu(B(x,r_k))\leq 10\,C_*\,\Lambda^n\,r^n.
\end{equation}
\end{lemma}

\begin{proof}
Denote by $E_0$ the subset of points $x\in E$ such that $\Theta^{n,*}(x,\mu)\leq 4C_*\,\Lambda^n$. 
Let $x \in E_0$ and consider a sequence of balls $B(x,r_k)$ with 
$\mu(B(x,\Lambda r_k)) \leq 2\Lambda^{d}\,\mu(B(x,r_k))$ (such sequence exists for $\mu$-a.e.\ $x\in E_0$,
as shown in Chapter of \cite{Tolsa-llibre}, for example).
It is clear then that \rf{eqrk0'} holds for $k$ big enough for $\mu$-a.e.\ $x\in E_0$.

In the case $x\in E\setminus E_0$, let $s_k\to 0$ be a sequence of radii such that 
$$\frac{\mu(B(x,s_k))}{s_k^n} \leq 2C_*.$$
Note that, for each $k$,
$$\limsup_{j\to\infty} \frac{\mu(B(x,\Lambda^{-j}s_k))}{(\Lambda^{-j}s_k)^n} \geq \Lambda^{-n} 
\limsup_{r\to0}\frac{\mu(B(x,r))}{r^n} \geq 4C_*.$$
Now we let $j\geq 0$ be the least integer such that
$$\frac{\mu(B(x,\Lambda^{-j}s_k))}{(\Lambda^{-j}s_k)^n} \geq 3C_*,
$$
and we set $r_k= \Lambda^{-j}s_k$.
Then we have
$$\mu(B(x,\Lambda r_k)) = 
\mu(B(x,\Lambda^{-j+1}s_k)) \leq 3 C_* (\Lambda^{-j+1}s_k)^n \leq \Lambda^n\,\mu(B(x,\Lambda^{-j}s_k))
=\Lambda^n\,\mu(B(x,r_k)),$$
which implies that $\mu(B(x,\Lambda r_k)) \leq 2\Lambda^{d}\,\mu(B(x,r_k))$ and also that
$$\mu(B(x,r_k))\leq \mu(B(x,\Lambda\,r_k))  \leq 3 C_* \Lambda^n\,r_k^n.$$
This concludes the proof of \rf{eqrk0'} for $\mu$-a.e.\ $x\in E$.
\end{proof}

\vv


\subsection{Proof of Theorem B}
Because of Theorem A,
it is enough to show that 
$$M_n\mu(x)=\sup_{r>0}\frac{\mu(B(x,r))}{r^n}<\infty\quad \mbox{ for $\mu$-a.e.\ $x\in E$.}$$
Recall that we  are assuming the conditions \rf{eqc*'} and \rf{eqjones*'}.

We need to consider an auxiliary approximating measure $\wt\mu$ which we proceed to define.
By Lemma \ref{lemdb*} and a Vitali type covering lemma, there is a family of pairwise disjoint 
balls $B_i$, $i\in I$,
centered at $E$, which cover $\mu$-a.e.\ $E$,
satisfying
\begin{equation}\label{eqrk0}
\mu( \Lambda B_i) \leq 2\Lambda^{d}\,\mu(B_i)
\quad \mbox{ and }\quad \mu(B_i)\leq 10\,C_*\,\Lambda^n\,r(B_i)^n,
\end{equation}
and also that 
$$r(B_i)\leq \rho,$$ for
some arbitrary fixed $\rho>0$.
Let $I_0\subset I$ be a finite subfamily such that
$$\mu\Bigl(E\setminus \bigcup_{i\in I_0} B_i\Bigr) \leq \ve\,\mu(E),$$
where $\ve>0$ is some small value to be chosen below.
For each $i\in I_0$,
we consider an $n$-dimensional disk $D_i$ concentric with $B_i$ and radius $\frac12 r(B_i)$
and 
 we define
$$\wt\mu = \sum_{i\in I_0} \frac{\mu(B_i)}{\HH^n(D_i)}\,\HH^n|_{D_i},$$
so that $\wt\mu(D_i) = \mu(B_i)$ for each $i\in I_0$.

We claim now that if $\Lambda$ is taken big enough, then
\begin{equation}\label{eqcalu821}
\int M_{n}\wt \mu \,d\wt\mu\leq C(\Lambda)\,C_*\,\mu(E) +
C\,\iint_0^\infty\beta_{\mu,2}(x,r)^2\,\frac{dr}
r\,d\mu(x),
\end{equation}
with the constants on the right hand side depending neither on $\rho$ nor on $\ve$.
Before proving \rf{eqcalu821}
 we show how this implies that $M_n\mu(x)<\infty$ $\mu$-a.e. 
 Indeed, by an approximating argument, and
 denoting 
$$M_{n,\rho}\mu(x) = \sup_{r>\rho}\frac{\mu(B(x,r))}{r^n}$$
and
$$E_{\ve,\rho} := E\cap \bigcup_{i\in I_0} B_i,$$ 
it follows easily that
\begin{align}\label{eqcalu822**}
\int_{E_{\ve,\rho}} M_{n,\rho} (\chi_{E_{\ve,\rho}}\mu) \,d\mu & \leq C\,\int M_{n,\rho}\wt \mu \,d\wt\mu.
\end{align}
To check this, take $x,x'\in B_j$, $j\in I_0$, and $r\geq \rho$. Then
\begin{align*}
\mu(B(x,r)\cap E_{\ve,\rho}) & \leq \mu(B(x',2r)\cap E_{\ve,\rho}) \leq \sum_{i\in I_0: B_i\cap B(x',2r)\neq\varnothing} \mu(B_i)\\
& = \sum_{i\in I_0: B_i\cap B(x',2r)\neq\varnothing} \wt\mu(D_i) \leq\wt\mu(B(x',3r)),
\end{align*}
taking into account that $B(x,r)\subset B(x',2r)$ in the first inequality, and that the balls $B_i$ in the
before to last sum are contained $B(x',3r)$. Therefore,
$$M_{n,\rho} (\chi_{E_{\ve,\rho}}\mu) (x) \leq 3^n\,\inf_{x'\in D_j} M_{n,\rho} \wt\mu (x')$$
for all $x\in B_j$, $j\in I_0$. The preceding estimate readily yields \rf{eqcalu822**} by integrating
with respect to $\mu$ in $E_{\ve,\rho}$.

From \rf{eqcalu821} and \rf{eqcalu822**} we get
\begin{align*}
\int_{E_{\ve,\rho}} M_{n,\rho} (\chi_{E_{\ve,\rho}}\mu) \,d\mu& \leq
C(\Lambda)\,C_*\,\mu(E) +
C\,\iint_0^\infty\beta_{\mu,2}(x,r)^2\,\frac{dr} 
r\,d\mu(x)=: K,
\end{align*}
with $K$ independent of $\rho$ and $\ve$. For $\rho>0$ fixed, take $\ve_k=2^{-k}$, and note that up to a set of null $\mu$-measure, $E=\liminf_k E_{\ve_k,\rho}$. Recall that by definition,
$$\liminf_k E_{\ve_k,\rho} = \bigcup_{j\geq 1} G_j,\quad \mbox{ with }\quad
G_j = \bigcap_{k\geq j} E_{\ve_k,\rho}.$$
Obviously, we have
\begin{equation*}
\int_{G_j} M_{n,\rho} (\chi_{G_j}\mu) \,d\mu\leq
\int_{E_{\ve_j,\rho}} M_{n,\rho} (\chi_{E_{\ve_j,\rho}}\mu) \,d\mu\leq K.
\end{equation*}
Since the sequence of sets $G_j$ is increasing, by monotone convergence we get
$$\chi_{G_j} M_{n,\rho} (\chi_{G_j}\mu)(x)\to  M_{n,\rho} \mu(x)
\quad \mbox{for $\mu$-a.e. $x\in E$.}$$ 
Then, again by monotone convergence, we deduce that
$\int M_{n,\rho}\mu \,d\mu\leq
 K$.
Since this estimate is uniform on $\rho$, again by monotone convergence we infer that
$$\int M_{n}\mu \,d\mu\leq
 K,$$
which shows that $M_n\mu(x)<\infty$ $\mu$-a.e., as wished.
\vv

It just remains to prove \rf{eqcalu821} now. To this end, we consider the corona decomposition associated to 
$\wt\mu$ described in Section \ref{sec41}. Notice that
the condition \rf{eqgrowth00'}
holds (with $C(\Lambda)\, C_*$ instead of $C_*$) for some $r_0>0$ because of the definition of $\wt\mu$, 
\rf{eqrk0}, and because the family $I_0$ is finite. Therefore, by \rf{eqpack00} and the subsequent discussion,
\begin{equation}\label{eqsum441'}
\sum_{R\in\ttt_{\wt\mu}}\Theta_{\wt\mu}(2B_R)\,\wt\mu(R)\leq C(\Lambda)\,C_{*}\,\wt\mu(\R^d)+
C\,\int\!\!\int_0^\infty\beta_{\wt\mu,2}(x,r)^2\,\frac{dr}r\,d\wt\mu(x).
\end{equation}
By the property (3) of the corona decomposition it is immediate to check that
$$\int M_{n}\wt \mu \,d\wt\mu\lesssim \sum_{R\in\ttt_{\wt\mu}}\Theta_{\wt\mu}(2B_R)\,\wt\mu(R),$$
and thus 
\begin{equation}\label{eqcalu822}
\int M_{n}\wt \mu \,d\wt\mu\lesssim  C(\Lambda)\,C_*\,\mu(E) +
C\int\!\!\int_0^\infty\beta_{\wt\mu,2}(x,r)^2\,\frac{dr}r\,d\wt\mu(x).
\end{equation}
Thus we just have to estimate the double integral on the right hand side.

Consider $x\in D_i$ for some $i\in I_0$ and
$r>0$. Note that $\beta_{\wt\mu,2}(x,r)=0$ unless $B(x,r)$ intersects some disc $D_j$, $j\neq i$.
In fact, denoting
$$D(B_i,B_j) = r(B_i)+r(B_j) + \dist(B_i,B_j),$$
by construction (using that the radius of $D_k$ is one half of the one of $B_k$),
$$\beta_{\wt\mu,2}(x,r)^2 = \inf_L \sum_{j\in I_0: D(B_i,B_j)\leq 2r} \int_{B(x,r)\cap B_j}
\biggl(\frac{\dist(y,L)}r\biggr)^2\,\frac{d\wt\mu(y)}{r^n}.$$
Observe also that the balls $B_i$ and $B_j$ appearing in this equation are contained in $B(y,20r)$
for all $y\in B_i$.
Then, taking into account that $\wt\mu(B_k)= \mu(B_k)$ for each $k\in I_0$, 
letting $L$ be the $n$-plane that minimizes $\beta_{\mu,2}(x,20r)$, for each $j$ in the sum above we have:
\begin{align*}
\int_{B_j}
\biggl(\frac{\dist(z,L)}r\biggr)^2\,\frac{d\wt\mu(z)}{r^n}
&\leq \int_{B_j}
\biggl(\frac{\sup_{z'\in B_j}\dist(z',L)}r\biggr)^2\,\frac{d\wt\mu(z)}{r^n}\\
& =
\int_{B_j}
\biggl(\frac{\sup_{z'\in B_j}\dist(z',L)}r\biggr)^2\,\frac{d\mu(z)}{r^n}\\
& 
\leq 2\int_{B_j}
\biggl(\frac{\dist(z,L)}r\biggr)^2\,\frac{d\mu(z)}{r^n} + 4\, \frac{r(B_j)^2}{r^{n+2}}\,\wt\mu(B_j).
\end{align*}
Hence for all $x\in D_i$ and $y\in B_i$ we can estimate $\beta_{\wt\mu,2}(x,r)$ in terms of $\beta_{\mu,2}(y,20r)$ as follows:
$$\beta_{\wt\mu,2}(x,r)^2 \lesssim 
\beta_{\mu,2}(y,20r)^2 + \sum_{j\in I_0: D(B_i,B_j)\leq 2r} \frac{r(B_j)^2}{r^{n+2}}\,\wt\mu(B_j).$$
So we obtain
\begin{align}\label{eqalp00}
\iint_0^\infty\beta_{\wt\mu,2}(x,r)^2\,\frac{dr}r\,d\wt\mu(x) 
& = \sum_{i\in I_0} \int_{0}^\infty\!\!\int_{B_i} \beta_{\wt\mu,2}(x,r)^2\,d\wt\mu(x)\,\frac{dr}r\\
& \lesssim \sum_{i\in I_0} \int_{0}^\infty\!\!\int_{B_i} \beta_{\mu,2}(y,20r)^2\,d\mu(y)\,\frac{dr}r\nonumber\\
&\quad+ \sum_{i\in I_0}\, \int_{0}^\infty\!\!\sum_{j\in I_0: D(B_i,B_j)\leq 2r} \frac{r(B_j)^2}{r^{n+2}}\,\wt\mu(B_j)
\,\frac{dr}r\,\wt\mu(B_i). \nonumber
\end{align}
Clearly, the first sum on the right hand side does not exceed $C\iint_0^\infty\beta_{\mu,2}(x,r)^2\,\frac{dr}r\,d\mu(x)$. Concerning the last sum, by Fubini this equals
\begin{equation}\label{eqplug53}
\sum_{j\in I_0}r(B_j)^2\,\wt\mu(B_j) \int_{0}^\infty\!\!\!\sum_{i\in I_0: D(B_i,B_j)\leq 2r} \!\!\wt\mu(B_i)
\,\frac{dr}{r^{n+3}}
\leq \sum_{j\in I_0}r(B_j)^2\,\wt\mu(B_j) \int_{r>r(B_j)/2}\! \!\frac{\wt\mu(B(x_j,20r))}{r^n}
\,\frac{dr}{r^{3}},
\end{equation}
where $x_j$ is the center of $D_j$ and $B_j$.
Now note that, for $0<r\leq \frac1{10}\,\Lambda \,r(B_j)$,
$$\frac{\wt\mu(B(x_j,20r))}{r^n}\leq C(\Lambda)\,\frac{\mu(B(x_j,20r))}{r^n}\leq C(\Lambda)\frac{\mu(B_j)}{r(B_j)^n}\leq C(\Lambda)\,C_*,$$
and also that
$$\frac{\wt\mu(B(x_j,20r))}{r^n}\leq C\,\inf_{y\in B_j} M_n\wt\mu(y)
\quad \mbox{for all $r>0$ and $y\in B_j$.}$$
Therefore,
\begin{align*}
\int_{r>r(B_j)/2}\! \frac{\wt\mu(B(x_j,20r))}{r^n}
\,\frac{dr}{r^{3}} &\leq \int_{r(B_j)/2}^{\frac1{10}\,\Lambda r(B_j)}  C(\Lambda)\,C_* \,\frac{dr}{r^{3}}+
\int_{\frac1{10}\,\Lambda r(B_j)}^\infty \inf_{y\in B_j} M_n\wt \mu(y)\,\frac{dr}{r^{3}}\\
&\lesssim \frac{C(\Lambda)\,C_*}{r(B_j)^2} + \frac1{\Lambda^2\,r(B_j)^2}\inf_{y\in B_j} M_n\wt\mu(y)
.
\end{align*}
Plugging this estimate into \rf{eqplug53}, we deduce that the last term on the right hand side of \rf{eqalp00} satisfies
\begin{align*}
\sum_{i\in I_0}\, \int_{0}^\infty\!\!\sum_{j\in I_0: D(B_i,B_j)\leq 2r} \frac{r(B_j)^2}{r^{n+2}}\,\wt\mu(B_j)
\,\frac{dr}r\,\wt\mu(B_i)
&\lesssim
\sum_{j\in I_0}C(\Lambda)\,C_*\,\wt\mu(B_j)\\
&\quad + \frac1{\Lambda^2}\sum_{j\in I_0}\wt\mu(B_j)\inf_{y\in B_j} M_n\wt\mu(y)\\
&\lesssim C(\Lambda)\,C_*\,\mu(E) + \frac1{\Lambda^2}\int M_n\wt\mu\,d\wt\mu.
 \end{align*}

From \rf{eqcalu822}, \rf{eqalp00}, and the preceding estimate we obtain
$$\int M_{n}\wt \mu \,d\wt\mu\lesssim 
C(\Lambda)\,C_*\,\mu(E) + \frac1{\Lambda^2}\int M_n\wt\mu\,d\wt\mu+  \iint_0^\infty\beta_{\mu,2}(x,r)^2\,\frac{dr}r\,d\mu(x).$$
Choosing $\Lambda$ big enough and taking into account that  $M_n\wt\mu\in L^\infty(\wt\mu)$ (by the construction of $\wt\mu$)
and that $\wt\mu$ is finite, we derive 
$$\int M_{n}\wt \mu \,d\wt\mu\leq 
C(\Lambda)\,C_*\,\mu(E) + C\iint_0^\infty\beta_{\mu,2}(x,r)^2\,\frac{dr}r\,d\mu(x),$$
as wished.

\vv

\vvv

\end{document}